\documentclass[12pt]{article}
\usepackage{amsmath,amssymb,amsthm, amsfonts}
\usepackage{hyperref}
\usepackage{graphicx}
\usepackage{url}
\usepackage[mathlines]{lineno}
\usepackage{dsfont} 
\usepackage{color}
\usepackage{mathtools}

\definecolor{red}{rgb}{1,0,0}

\definecolor{blue}{rgb}{0,0,1}

\definecolor{green}{rgb}{0,.6,0}

\newtheorem{thm}{Theorem}[section]
\newtheorem{cor}[thm]{Corollary}
\newtheorem{lem}[thm]{Lemma}

\theoremstyle{definition}

\theoremstyle{definition}

\theoremstyle{definition}

\newcommand{\thp}{\operatorname{th_+}}

\newcommand{\ptp}{\operatorname{pt_+}}

\newcommand{\rad}{\operatorname{rad}} 
\newcommand{\thc}{\operatorname{th_c}}
\newcommand{\ectu}{\operatorname{ect_u}}
\newcommand{\ecto}{\operatorname{ect_1}}

\newcommand{\ectv}{\operatorname{ect_v}}

\newcommand{\thg}{\operatorname{th_{A,\gamma}}}

\newcommand{\capt}{\operatorname{capt}}
\newcommand{\cn}{\operatorname{c}}

\newcommand{\ectuk}{\operatorname{ect_{u,k}}}
\newcommand{\ectok}{\operatorname{ect_{1,k}}}
\newcommand{\ectvk}{\operatorname{ect_{v,k}}}

\newcommand{\thu}{\operatorname{th_u}}
\newcommand{\tho}{\operatorname{th_1}}

\newcommand{\thv}{\operatorname{th_v}}

\newcommand{\captk}{\operatorname{capt_k}}

\newcommand{\bit}{\begin{itemize}}
\newcommand{\eit}{\end{itemize}}
\newcommand{\ben}{\begin{enumerate}}
\newcommand{\een}{\end{enumerate}}
\newcommand{\beq}{\begin{equation}}
\newcommand{\eeq}{\end{equation}}
\newcommand{\bea}{\begin{eqnarray*}} 
\newcommand{\eea}{\end{eqnarray*}}
\newcommand{\bpf}{\begin{proof}}
\newcommand{\epf}{\end{proof}\ms}
\newcommand{\bmt}{\begin{bmatrix}}
\newcommand{\emt}{\end{bmatrix}}
\newcommand{\ms}{\medskip}

\newcommand{\lc}{\left\lceil}
\newcommand{\rc}{\right\rceil}
\newcommand{\lf}{\left\lfloor}
\newcommand{\rf}{\right\rfloor}

\title{Throttling numbers for adversaries on connected graphs}
\author{Jesse Geneson}

\begin{document}
\maketitle

\begin{abstract} 
In this paper, we answer two open problems from [Breen et al., Throttling for the game of Cops and Robbers on graphs, Discrete Math., 341 (2018) 2418–2430]. The throttling number $\thc(G)$ of a graph $G$ is the minimum possible value of $k + \captk(G)$ over all positive integers $k$, where $\captk(G)$ is the number of rounds needed for $k$ cops to capture the robber on $G$. One of the problems from [Breen et al., 2018] was to determine whether there exists a family of trees $T$ of order $n$ for which $\thc(T)$ is asymptotically equal to $2 \sqrt{n}$. We show that such a family cannot exist by improving the upper bound on $\displaystyle \max_{T} \thc(T)$ for all trees $T$ of order $n$ from $2 \sqrt{n}$ to $\frac{\sqrt{14}}{2} \sqrt{n} + O(1)$. We prove this bound by deriving a more general throttling bound for connected graphs that applies to multiple graph adversaries, including the robber and the gambler. This also improves the best known upper bounds on $\thc(G)$ for chordal graphs and unicyclic graphs $G$, as well as throttling numbers for positive semidefinite (PSD) zero forcing on trees. In addition to the results about cop versus robber, we use our general throttling bound to improve previous upper bounds on throttling numbers for the cop versus gambler game on connected graphs. 

Another open problem from [Breen et al., 2018] was to obtain a bound on $\thc(G)$ for cactus graphs $G$. We prove an $O(\sqrt{n})$ bound for all cactus graphs $G$ of order $n$. Furthermore, we exhibit a family of trees $T$ of order $n$ that have $\thc(T) > 1.4502 \sqrt{n}$ for all $n$ sufficiently large, improving on the previous lower bound of $\lc \sqrt{2n}-\frac{1}{2} \rc + 1$ on $\displaystyle \max_{T} \thc(T)$ for trees $T$ of order $n$.
\end{abstract}

\section{Introduction}

In the \emph{cop versus robber} game on a graph $G$, a team of cops choose their initial positions at vertices of $G$, followed by the robber. Multiple cops are allowed to occupy the same vertex. In each round, every cop is allowed to stay at their current vertex or move to a neighboring vertex, and then the robber is allowed to move analogously. The cops try to occupy the same vertex as the robber (which represents capture), while the robber tries to avoid being captured. 

Past research on this game (see e.g. \cite{af, nw, pr, qu}) has focused on the \emph{cop number} $\cn(G)$, which is the minimum possible number of cops that can capture the robber on $G$ if the cops and robber play optimally, as well as the \emph{capture time} $\capt(G)$, which is the number of rounds that it takes for $\cn(G)$ cops to capture the robber on $G$ if the cops and robber play optimally. A famous conjecture about the cop number is that $\cn(G) = O(\sqrt{n})$ for any connected $n$-vertex graph $G$ (known as Meyniel's conjecture) \cite{frankl}. More recently, a different parameter was introduced to study the cop versus robber game \cite{cr1}. Let $\captk(G)$ be the number of rounds for $k$ cops to capture the robber on $G$ if the cops and robber play optimally, and define the \emph{throttling number} $\thc(G) = \displaystyle  \min_{k} (k + \captk(G))$.

It was proved in \cite{cr1} that on trees $T$, the values of $\thc(T)$ are equal to the throttling numbers for a graph coloring process called \emph{positive semidefinite (PSD) zero forcing}. In this process, an initial set of vertices is colored blue. In each round of the process, $S$ denotes the set of blue vertices and $W_1, \dots, W_r$ denote the sets of white vertices corresponding to the connected components of $G - S$. A blue vertex $v$ colors one of its white neighbors $w$ blue if for some $i$, $w$ is the only white neighbor of $v$ in the subgraph of $G$ induced by $W_i \cup S$. The parameter $\ptp(G, k)$ is the minimum possible number of rounds for all of $G$ to be colored blue when the set of initial blue vertices has size $k$, and the \emph{PSD throttling number} is defined by letting $\thp(G) = \displaystyle \min_{k} \left\{ k + \ptp(G, k) \right\}$ \cite{psdth}.

Using the fact that $\thc(T) = \thp(T)$ for all trees $T$, it was shown in \cite{cr1} that $\thc(T) \leq 2 \sqrt{n}$ for every tree $T$ of order $n$. This was generalized in \cite{cr2}, where it was shown that $\thc(G) \leq 2 \sqrt{n}$ for every chordal graph $G$ of order $n$, and $\thc(H) \leq 2 \sqrt{n}+k$ for every connected graph $H$ of order $n$ with $k$ cycles. In this paper, we sharpen all of these bounds by proving that $\thc(T)  \leq \frac{\sqrt{14}}{2} \sqrt{n} + O(1)$ for every tree $T$ of order $n$. This gives improved bounds for chordal graphs and connected graphs with $k$ cycles as a corollary. Using a similar method we also show that $\thc(S)  \leq \sqrt{3} \sqrt{n} + O(1)$ for every $n$-vertex spider $S$, where a spider is a tree with only a single vertex of degree more than $2$. A spider can also be viewed as a graph obtained from a collection of disjoint paths by adding a single vertex that is adjacent to one endpoint from each path. Each path is called a \emph{leg} of the spider, and the length of a leg is the number of vertices in its path. We also prove that $\thc(G) = O(\sqrt{n})$ for cactus graphs $G$ of order $n$, answering an open problem from \cite{cr1}, where cactus graphs are connected graphs in which any two simple cycles have at most one vertex in common.

Our technique for proving these new bounds is very different from previous techniques used to throttle the robber on trees. Moreover, the method that we use to bound robber throttling also works for other adversaries. The \emph{gambler} is an adversary introduced by Komarov and Winkler \cite{kw}. Given a graph $G$ with vertices $v_1, \dots, v_n$, the gambler uses a probability distribution $p_1, \dots, p_n$ to choose the vertex that it will occupy on every turn. In each round, all cops and the gambler move simultaneously. The whole game is played in darkness. The gambler is called \emph{known} if the cops know their distribution, otherwise the gambler is called \emph{unknown}. If the cops observe $k$ gambler turns before the game begins, the gambler is called \emph{$k$-observed} \cite{gthr}. 

Komarov and Winkler \cite{kw} proved that the known gambler has expected capture time $\ectv(G) = n$ on any connected graph $G$ of order $n$, assuming that the cops and gambler play optimally. They also proved an upper bound of approximately $1.97n$ on the expected capture time $\ectu(G)$ of the unknown gambler on any connected graph $G$ of order $n$, which was later improved to approximately $1.95n$ in \cite{unkg}. It was shown in \cite{gthr} that the $1$-observed gambler has expected capture time $\ecto(G) \leq 1.5n$. A strategy for multiple cops to pursue the unknown gambler was introduced in \cite{distg}.

We denote expected capture times versus $k$ cops as $\ectvk(G)$ for the known gambler, $\ectuk(G)$ for the unknown gambler, and $\ectok(G)$ for the $1$-observed gambler. Analogously to the robber, we define the known, unknown, and $1$-observed gambler throttling numbers respectively to be $\thv(G) = \displaystyle  \min_{k} (k + \ectvk(G))$, $\thu(G) = \displaystyle \min_{k} (k + \ectuk(G))$, and $\tho(G) = \displaystyle \min_{k} (k + \ectok(G))$. Throttling numbers for the gambler were previously bounded in \cite{gthr}, where it was shown for any connected graph $G$ of order $n$ that $\thv(G) \leq \tho(G) \leq 2 \sqrt{3n} + O(1)$ and $\thu(G) < 3.96944 \sqrt{n}$ for sufficiently large $n$. We improve the gambler throttling upper bounds by proving for every connected graph $G$ of order $n$ that $\tho(G)  \leq \frac{\sqrt{42}}{2} \sqrt{n} + O(1)$ and $\thu(G) < 3.7131 \sqrt{n}$ for all $n$ sufficiently large. 

In addition to the upper bounds, we also prove a lower bound on $\displaystyle \max_{T} \thc(T)$ over all trees $T$ of order $n$. Previously the sharpest known lower bound for this quantity was $\lc \sqrt{2n}-\frac{1}{2} \rc + 1$ \cite{ross}, which exceeds $\thc(P_n) = \lc \sqrt{2n}-\frac{1}{2} \rc$ by $1$. We construct a family of spiders that give a lower bound of $1.4502 \sqrt{n}$ for all $n$ sufficiently large. This is the first known family of trees of order $n$ whose throttling numbers grow arbitrarily larger than $\thc(P_n)$.

In Section \ref{generalbound}, we prove the general throttling bound for graph adversaries. In Section \ref{robber}, we apply the general bound to obtain upper bounds on robber throttling numbers for trees, chordal graphs, and graphs with $o(\sqrt{n})$ cycles. In Section \ref{gam}, we apply the general bound to obtain upper bounds on gambler throttling numbers. In Section \ref{cactis}, we prove an $O(\sqrt{n})$ bound on $\thc(G)$ for cactus graphs $G$ of order $n$, answering an open problem from \cite{cr1}. In the same section, we also prove a $\sqrt{3} \sqrt{n} + O(1)$ upper bound on $\thc(S)$ for every $n$-vertex spider $S$ and an improved lower bound on $\displaystyle \max_{T} \thc(T)$ over all trees $T$ of order $n$.

\section{General throttling bound for adversaries on connected graphs}\label{generalbound}

The first lemma in this section was introduced in an unpublished manuscript of the author about the cop versus gambler game \cite{distg} and was also used in \cite{gthr} to bound expected capture time and throttling numbers for different types of gamblers. We include the proof for completeness. For the proof, we define a \emph{limb} of a rooted tree $T$ as some vertex $w$ in $T$ together with some number of branches rooted at $w$ and below $w$ in $T$.

\begin{lem}\label{treecov} \cite{distg, gthr}
If $T = (V,E)$ is a rooted tree with $|V| = n$ and $x$ is a real number such that $1.5 \leq x < n$, then there exists a subset $S \subset V$ and a vertex $v \in S$ such that $x < |S| \leq 2x-1$, $T|_{S}$ is a limb, and $T|_{(V-S) \cup \left\{v\right\} }$ is connected.
\end{lem}

\begin{proof}
By induction on $n$: Suppose that $u$ is the root of $T$. If $x < n \leq 2x-1$, set $S = V$ and $v = u$. For $n > 2x-1$, the inductive hypothesis is that the lemma is true for all $m$ such that $m < n$. Let $B$ be a branch of $T$ rooted at $u$ with a maximum number of vertices among all branches rooted at $u$. 

If $|B| > 2x-1$, the inductive hypothesis on $B - \left\{u \right\}$ gives values for $S$ and $v$. If $x < |B| \leq 2x-1$, then set $S = B$ and $v = u$. If $|B| \leq x$, set $C = B$ and add all vertices to $C$ from each next largest branch rooted at $u$ until $C$ has more than $x$ vertices. Since $B$ has the maximum number of vertices among all branches rooted at $u$, this process results in $x < |C| \leq 2x-1$. Set $S = C$ and $v = u$.
\end{proof}

\begin{cor}\label{treecov3}
If $T = (V, E)$ is a tree with $|V| = n$, then there exist subsets $S_0, S_1 \subset V$ such that $T|_{S_i}$ is connected for each $i$, $S_0 \cup S_1 = V$, $|S_0 \cap S_1| \leq 1$, and $\frac{n}{3} < |S_0| < \frac{2n}{3}$. 
\end{cor}

\begin{cor}\label{treecovm}
If $T = (V, E)$ is a tree with $|V| = n$ and $x$ is a real number such that $1.5 \leq x < n$, then there exists a positive integer $s$ and disjoint subsets $Y_0, \dots, Y_s\subset V$ such that $\bigcup\limits_{i=0}^{s} Y_i = V$, $x-1 < |Y_i| \leq 2x-1$ for $0 \leq i \leq s-1$, $|Y_s| \leq x$, and for each $i$ there exists a vertex $v_i \in V$ such that $T|_{Y_i \cup \left\{v_i\right\}}$ is connected.
\end{cor}

\begin{proof}
Let $T_0 = T$. Apply Lemma \ref{treecov} to $T_0$ to obtain a subset $S_0 \subset V$ and a vertex $v_0 \in S_0$ such that $x < |S_0| \leq 2x-1$, $T_0|_{S_0}$ is connected, and $T_0|_{(V-S_0) \cup \left\{v_0\right\}}$ is connected. Let $Y_0 = S_0 - \left\{v_0\right\}$ and let $T_1 = T_0 - Y_0$. Now suppose for inductive hypothesis that $Y_0, \dots, Y_j$ have been chosen so far, and that $T_{j+1} = T_j - Y_j$. If $|T_{j+1}| > x$, then apply Lemma \ref{treecov} to $T_{j+1}$ to obtain a subset $S_{j+1} \subset V(T_{j+1})$ and a vertex $v_{j+1} \in S_{j+1}$ such that $x < |S_{j+1}| \leq 2x-1$, $T_{j+1}|_{S_{j+1}}$ is connected, and $T_{j+1}|_{(V(T_{j+1})-S_{j+1}) \cup \left\{v_{j+1}\right\}}$ is connected. Let $Y_{j+1} = S_{j+1} - \left\{v_{j+1}\right\}$ and let $T_{j+2} = T_{j+1} - Y_{j+1}$. Otherwise if $|T_{j+1}| \leq x$, then $s = j+1$ and $Y_s = V(T_s)$.
\end{proof}

To prove the next lemma, we define some terminology. We use $A$ to denote an adversary on graphs, such as the robber or unknown gambler. Given an adversary $A$, let $\gamma_A$ be a time parameter associated with capturing $A$ using any positive number of cops. For example, $\gamma_A$ could be capture time when $A$ is the robber and expected capture time when $A$ is the unknown gambler. We use $\gamma_{A,k}(G)$ to refer to the value of $\gamma_A$ associated with capturing $A$ using $k$ cops on the graph $G$. The $\gamma$-throttling number of $A$ on the graph $G$ is $\thg(G) = \displaystyle  \min_{k} (k+\gamma_{A,k}(G))$.

\begin{lem}\label{genadv}
Suppose that $A$ is an adversary, $G$ is a connected graph of order $n$, and $c$ is a positive real number such that $\gamma_{A,k}(G) \leq c x + O(1)$ for every pair of positive integers $x, k$ for which $V(G)$ can be covered with $k$ connected subgraphs of order at most $x$. Then $\thg(G) \leq \sqrt{7 c} \sqrt{n} + O(1)$.
\end{lem}

\begin{proof} 
Define $r = \sqrt{\frac{4}{7c}}$ and let $T$ be a spanning subtree of $G$. By Corollary \ref{treecovm}, the vertices of $T$ can be partitioned into disjoint subsets $Y_0, \dots, Y_s$ such that $r\sqrt{n}-1 < |Y_i| < 2r \sqrt{n}$ for $0 \leq i \leq s-1$, $|Y_s| \leq r \sqrt{n}$, and for each $0 \leq i \leq s$ there is some vertex $v_i$ in $T$ such that $T|_{Y_i \cup \left\{v_i \right\}}$ is connected. Let $b$ be the number of subsets $Y_i$ of size more than $1.5 (r \sqrt{n}-1)$ for $0 \leq i \leq s-1$. If $b > \frac{r c}{2}\sqrt{n}$, then let $k = 1+s$. In this case, $\gamma_{A,k}(G) \leq 2 r c \sqrt{n} + O(1)$ and $k \leq 1+\frac{n-1.5 b (r \sqrt{n}-1)}{r \sqrt{n}-1}+ b < \frac{1}{r}\sqrt{n}-\frac{b}{2} +2+\frac{1}{r^2} < \frac{3\sqrt{7c}}{7}\sqrt{n}+2+\frac{1}{r^2}$ for all $n$ sufficiently large. Thus in this case, $\thg(G) \leq \sqrt{7c} \sqrt{n} + O(1)$.

If $b \leq \frac{r c}{2} \sqrt{n}$, then let $k = 1+s+b$. For the subsets $Y_i$ of size more than $1.5( r \sqrt{n}-1)$, use Corollary \ref{treecov3} to replace $T|_{Y_i \cup \left\{v_i\right\}}$ in the cover of $V(G)$ with two subtrees $R_i$ and $S_i$ that each have order between $\frac{|Y_i|+1}{3}$ and $\frac{2(|Y_i|+1)}{3}+1$. In this case, $\gamma_{A,k}(G) \leq \frac{3}{2} r c \sqrt{n} + O(1)$ and $k \leq 1+\frac{n-1.5 b (r \sqrt{n}-1)}{r \sqrt{n}-1}+2b < \frac{1}{r}\sqrt{n}+\frac{b}{2} + 2 + \frac{1}{r^2} \leq \frac{4\sqrt{7c}}{7}\sqrt{n}+2+\frac{1}{r^2}$ for all $n$ sufficiently large. Thus also in this case, $\thg(G)  \leq \sqrt{7c} \sqrt{n} + O(1)$.
\end{proof}

\section{Improved upper bounds for the robber and gamblers}

In this section, we sharpen the upper bounds on $\thc(T)$ for trees $T$ of order $n$ from $2 \sqrt{n}$ to $\frac{\sqrt{14}}{2} \sqrt{n} + O(1)$. As corollaries, we obtain upper bounds on $\thp(T)$ for trees $T$ of order $n$, $\thc(G)$ for chordal graphs $G$ of order $n$, and $\thc(G)$ for connected graphs $G$ of order $n$ with a bounded number of cycles. We also improve the throttling upper bounds for the known, unknown, and $1$-observed gamblers. 

\subsection{Bounds for the robber}\label{robber}

The first result in this section is a corollary of Lemma \ref{genadv} and the fact that the radius of a tree is at most half its order.

\begin{thm}\label{uppertree}
For every tree $T$ of order $n$, $\thc(T) \leq \frac{\sqrt{14}}{2} \sqrt{n} +O(1)$.
\end{thm}

\begin{proof}
If $V(T)$ can be covered with $k$ connected subgraphs of order at most $x$, then $k$ cops can capture the robber in at most $\frac{1}{2} x$ rounds by starting at the center of each subgraph and moving toward the robber in each round. Thus Lemma \ref{genadv} can be applied with $c = \frac{1}{2}$.
\end{proof}

\begin{cor}
For every tree $T$ of order $n$, $\thp(T) \leq \frac{\sqrt{14}}{2} \sqrt{n} +O(1)$.
\end{cor}

\begin{cor}
For every chordal graph $G$ of order $n$, $\thc(G)  \leq \frac{\sqrt{14}}{2} \sqrt{n} + O(1)$.
\end{cor}

\begin{proof}
For every chordal graph $G$, it is known that $\thc(G) = \displaystyle \min_{k}(k+\rad_k (G))$, where $\rad_k(G)$ denotes the $k$-radius of $G$ \cite{cr2}. Thus this corollary follows by applying Theorem \ref{uppertree} to any spanning subtree of $G$.
\end{proof}

In \cite{cr2}, it was shown that if $\thc(T) \leq c \sqrt{n}$ for all trees $T$ of order $n$, then $\thc(G) \leq c \sqrt{n}+k$ for all connected graphs $G$ of order $n$ with at most $k$ cycles. Together with Theorem \ref{uppertree}, this implies the next corollary. In particular, this sharpens the upper bound for throttling the robber on unicyclic graphs from \cite{cr1} and \cite{cr2}.

\begin{cor}
If $G$ is a connected graph of order $n$ with at most $k = o(\sqrt{n})$ cycles, then $\thc(G)  \leq (\frac{\sqrt{14}}{2}+o(1)) \sqrt{n}$. 
\end{cor}

\subsection{Bounds for the gamblers}\label{gam}

In \cite{distg}, it was shown that if $G$ is a connected graph and $V(G)$ is covered with connected subgraphs of order at most $x$, then there is a strategy to place one cop in each subgraph such that the expected capture time of the unknown gambler is at most $3(\frac{1}{1-\frac{1}{e^2}}-\frac{1}{2})x + O(1)$. In \cite{gthr}, a strategy was provided to show that the expected capture time of the $1$-observed gambler is at most $\frac{3x}{2} + O(1)$ in the same scenario. Thus the next two results follow from applying Lemma \ref{genadv} with $c = 3(\frac{1}{1-\frac{1}{e^2}}-\frac{1}{2})$ and $c = \frac{3}{2}$ respectively.

\begin{thm}\label{treeunk}
For every connected graph $G$ of order $n$, $\thu(G) < 3.7131 \sqrt{n}$ for all $n$ sufficiently large. 
\end{thm}

\begin{thm}\label{treeobs}
For every connected graph $G$ of order $n$, $\tho(G)  \leq \frac{\sqrt{42}}{2} \sqrt{n} + O(1)$. 
\end{thm}

\section{Cacti and spiders}\label{cactis}

We first prove that $\thc(G) = O(\sqrt{n})$ for cactus graphs $G$ of order $n$. This upper bound is tight up to a constant factor, and solves a problem from \cite{cr1}.

\begin{thm}
For every cactus graph $G$ of order $n$, $\thc(G) = O(\sqrt{n})$.
\end{thm}

\begin{proof}
Given any cactus graph $G$, we construct a tree $T_G$ with the property that each vertex in $T_G$ corresponds to at most two vertices in $G$. The tree $T_G$ is obtained by flattening every simple cycle in $G$ into a path: Start with any simple cycle $C$ in $T_G$. If $C$ has $2k$ vertices labelled $v_1, \dots, v_{2k}$ in cyclic order, it is flattened into a path on $k+1$ vertices with endpoints $v_k$ and $v_{2k}$, and vertices $v_{k-i}$ and $v_{k+i}$ contracted into one vertex for each $i = 1, \dots, k-1$. If $C$ has $2k+1$ vertices labelled $v_1, \dots, v_{2k+1}$ in cyclic order, it is flattened into a path on $k+2$ vertices. In this case, the endpoints are still $v_k$ and $v_{2k}$, and vertices $v_{k-i}$ and $v_{k+i}$ are still contracted into one vertex for each $i = 1, \dots, k-1$, but $v_{2k+1}$ appears between $v_{2k}$ and the contraction of $v_1$ and $v_{2k-1}$ on the path. Next, we repeat the same process with every simple cycle $C'$ that shares some common vertex $v$ with $C$, except we label the vertices of $C'$ so that $v$ becomes an endpoint of the path obtained from $C'$. We treat edges not part of any simple cycle in $G$ as $2$-cycles, so any edge adjacent to $v$ stays the same when the flattening process is performed on it. We continue this process until all of the simple cycles are flattened, at each step flattening all of the simple cycles $C'$ (including $2$-cycles) that share a common vertex $v$ with a simple cycle that was flattened in the previous round, in such a way that $v$ becomes an endpoint of the path obtained from $C'$. 

Note that at each step of the construction, the only time that multiple vertices of $G$ are contracted into one vertex of $T_G$ is when a simple cycle $C$ is flattened, and each of these contractions change two vertices in $G$ into one vertex in $T_G$. Thus each vertex in $T_G$ corresponds to at most two vertices in $G$.

Since $T_G$ has order at most $n$, we can use Corollary \ref{treecovm} to cover $T_G$ with at most $\sqrt{n}+1$ subtrees of order at most $2 \sqrt{n}+2$. By the proof of Lemma \ref{treecov}, each of these subtrees has a distinguished vertex $v$, and the distinguished vertices are the only vertices where the subtrees connect, so we place a cop on the vertices in $G$ corresponding to each of these distinguished vertices in $T_G$, using at most $2\sqrt{n}+2$ cops. This forces the robber to stay in a subgraph of $G$ that corresponds to a single subtree of $T_G$ for the entire game. 

Each subtree of $T_G$ corresponds to a planar subgraph with at most two connected components in $G$, and Pisantechakool and Tan \cite{pt} proved that $\capt_3(H) \leq 2n$ for any connected planar graph $H$ of order $n$. Thus in addition to the cops that guard the distinguished vertices, we can use $6$ cops per subtree to capture the robber in at most $8 \sqrt{n}+8$ rounds. Thus $\thc(G) \leq 16 \sqrt{n}+16$.
\end{proof}

Next, we use a similar technique to Lemma \ref{genadv} to get a sharper upper bound for the family of spider graphs. Ross proved that $\thc(S) \leq \lc \sqrt{2n}-\frac{1}{2} \rc$ for balanced spiders $S$ of order $n$, i.e., spiders with all legs of equal length \cite{ross}. The next result applies to all spiders, including those with legs of different sizes.

\begin{thm}\label{upperspider}
For every spider $S$ of order $n$, $\thc(S)  \leq \sqrt{3} \sqrt{n} + O(1)$.
\end{thm}

\begin{proof}
Define $r = \sqrt{\frac{4}{3}}$. Cut intervals of length $\lf r \sqrt{n} \rf$ off of the spider's legs until the remaining legs all have length less than $\lf r \sqrt{n} \rf$. Let $b$ be the number of modified legs of length more than $\frac{1}{2} (r \sqrt{n}-1)$. 

If $b > \frac{r}{2}\sqrt{n}$, assign one cop to every interval of length $\lf r \sqrt{n} \rf$ and one cop to the center. In this case, the capture time is at most $r \sqrt{n}$ and the number of cops is at most $1+\frac{n- \frac{1}{2} b (r \sqrt{n}-1)}{r \sqrt{n}-1} < \frac{1}{r}\sqrt{n}-\frac{b}{2} +2 < \frac{\sqrt{3}}{3}\sqrt{n}+2$ for all $n$ sufficiently large. Thus in this case, $\thc(S) \leq \sqrt{3} \sqrt{n} + O(1)$.

If $b \leq \frac{r}{2} \sqrt{n}$, assign one cop to every interval of length $\lf r \sqrt{n} \rf$, one cop to the center, and one cop to every interval with length between $\frac{1}{2} (r\sqrt{n}-1)$ and $\lf r \sqrt{n} \rf$. In this case, the capture time is at most $\frac{1}{2} r \sqrt{n}$ and the number of cops is at most $1+\frac{n-\frac{1}{2} b (r \sqrt{n}-1)}{r \sqrt{n}-1}+b < \frac{1}{r}\sqrt{n}+\frac{b}{2} + 2 \leq \frac{2\sqrt{3}}{3}\sqrt{n}+2$ for all $n$ sufficiently large. Thus also in this case, $\thc(S) \leq \sqrt{3} \sqrt{n} + O(1)$.
\end{proof}

Finally, we turn to proving a sharper lower bound on $\displaystyle \max_{T} \thc(T)$ over all trees $T$ of order $n$. Prior to this lower bound, the sharpest known lower bound for $\displaystyle \max_{T} \thc(T)$ over all trees $T$ of order $n$ was $\lc \sqrt{2n}-\frac{1}{2} \rc + 1$ \cite{ross}. To prove the new lower bound, we construct a family of unbalanced spiders. For real numbers $0 < a, c < 1$, define $S_{n, a, c}$ to be the spider of order $n$ with $\lf a \sqrt{n} \rf$ short legs of length $\lf c \sqrt{n} \rf$ and one long leg of length $n-1-\lf a \sqrt{n} \rf\lf c \sqrt{n} \rf$.

\begin{thm}\label{treelower}
There exist trees $T$ of order $n$ with $\thc(T) > 1.4502 \sqrt{n}$ for all $n$ sufficiently large. 
\end{thm}

\begin{proof}
Consider the spider $S_{n, a, c}$ with $c = 1.08766$ and $a = \frac{3c+2c^3-\sqrt{48c^4-32c^6}}{9c^2}$. We split the proof into two cases. 

If there is any short leg with no initial cops, then the capture time is at least $\lf c \sqrt{n} \rf$. For any $x \geq \lf c \sqrt{n} \rf$, the long leg must have at least $\frac{n-1-\lf a \sqrt{n} \rf\lf c \sqrt{n} \rf-x}{2x+1}$ cops to have capture time at most $x$, so in this case the sum of the capture time and the number of cops is at least $\displaystyle \min_{x \geq \lf c \sqrt{n} \rf} (x+\frac{n-1-\lf a \sqrt{n} \rf\lf c \sqrt{n} \rf-x}{2x+1}) = \lf c \sqrt{n} \rf + \frac{n-1-(\lf a \sqrt{n} \rf+1)\lf c \sqrt{n} \rf}{2\lf c \sqrt{n} \rf+1} > 1.4502 \sqrt{n}$ for all $n$ sufficiently large. 

If every short leg has an initial cop, then there are at least $\lf a \sqrt{n} \rf$ cops on short legs. Since $\thc(P_n) = \lc \sqrt{2n}-\frac{1}{2} \rc$, in this case the sum of the capture time and the number of cops on $S_{n, a, c}$ is at least $\lc \sqrt{2(n-1-\lf a \sqrt{n} \rf\lf c \sqrt{n} \rf)}-\frac{1}{2} \rc-1 + \lf a \sqrt{n} \rf > 1.4502\sqrt{n}$ for all $n$ sufficiently large.
\end{proof}

\begin{cor}
There exist trees $T$ of order $n$ with $\thp(T) > 1.4502 \sqrt{n}$ for all $n$ sufficiently large. 
\end{cor}

It is possible to extend the lower bound in Theorem \ref{treelower} for throttling the robber to a more general class of adversaries. Suppose that $A$ is an adversary that chooses their initial position after the cops and is able to stay at one vertex for the entire game. Then for all $n$ sufficiently large, there exist trees $T$ of order $n$ such that for every positive integer $k$, $A$ can avoid being captured by $k$ cops in time under $1.4502 \sqrt{n}-k$. This is because after the $k$ cops choose their initial positions, $A$ can choose an initial position that has maximum possible distance from the set of initial cop vertices. After choosing an initial position, $A$ does not move for the remainder of the game.

\section*{Acknowledgement}

The author thanks Leslie Hogben for helpful suggestions on this paper.

\end{document}